\documentclass[11pt,a4paper]{article}

\usepackage{amsmath}
\usepackage{amsfonts}
\usepackage{amsthm}
\usepackage[pdftex]{color,graphicx}
\usepackage[colorlinks,linkcolor=black,citecolor=black,urlcolor=black]{hyperref}

\newtheorem{Theorem}{Theorem}

\newtheorem{Proposition}{Proposition}
\newtheorem{Lemma}{Lemma}
\newtheorem{Remark}{Remark}
\newtheorem{Example}{Example}

\newcommand{\dt}{{\mathrm{d}}t}
\newcommand{\ds}{{\mathrm{d}}s}
\newcommand{\dx}{{\mathrm{d}}x}
\newcommand{\Id}{{\mathbf{1}}}
\newcommand{\dom}{{\mathrm{dom}~}}

\begin{document}
	
\title{\bf\Large EXISTENCE OF DISCRETE EIGENVALUES FOR THE DIRICHLET LAPLACIAN IN A TWO-DIMENSIONAL TWISTED STRIP}
\author{\bf \normalsize RAFAEL T. AMORIM\thanks{The first author was supported by CNPq, Brazil (grant 141842/2019-9).}\;  and ALESSANDRA A. VERRI}

%\title{\sc Existence of discrete eigenvalues for the Dirichlet Laplacian in a two-dimensional twisted strip}
%\author{Rafael T. Amorim\thanks{The first author was supported by CNPq, Brazil (grant 141842/2019-9).}\;  and Alessandra A. Verri}

\date{\today}

\maketitle 

%twisted effect locally slows down

\begin{abstract}
We study the spectrum of the Dirichlet Laplacian operator in
twisted strips on ruled surfaces in any space dimension.
It is shown that a suitable twisted effect can create discrete eigenvalues for the operator.
In particular, we also study the case where the twisted effect
``grows'' at infinity while the width of the strip goes to zero. In this situation, 
we find an asymptotic behavior for the eigenvalues.
\end{abstract}

\


\


%\

%\noindent {Running head:} XXXXXXXXXXXXXX

%\

\section{Introduction}

Let $\Omega$ be a strip on a surface in $\mathbb R^d$, $d \geq 2$, and denote by $-\Delta_\Omega^D$ 
the Dirichlet Laplacian operator in $\Omega$. 
A problem extensively studied in the literature is to find spectral information of $-\Delta_\Omega^D$. 
If $\Omega$ is a bounded strip, it is known that the spectrum $\sigma(-\Delta_\Omega^D)$ is 
purely discrete. 
Otherwise, 
the existence of discrete eigenvalues is a non-trivial property and it
depends on the 
geometry of $\Omega$
\cite{bori1, bori2, duclosfk, clark, duclos, pavelduclos,  
exnerkovarik, bookexner, solomyak, friedsol, friedlander, davidsurf, davidhardy, davidruled, davidsedi, renger}.

%Consider the case where $\Omega$ is an unbounded strip obtained as 
%a tubular neighborhood of constant width along an infinite curve in $\mathbb R^2$. More precisely,
At first, consider the following class of planar strips.
Let $r: \mathbb R \to \mathbb R^2$ be a $C^\infty$ curve parameterized by its arc-length $s$
and denote by $k(s)$ its curvature at the point $r(s)$. Consider the case where
$\Omega$ is obtained by moving a bounded 
segment $(c,d)$ along $r$ with respect to its normal vector field.
In the pioneering paper \cite{exnerseba},
the authors studied the spectral problem of $-\Delta_\Omega^D$. In particular, 
they proved the existence of discrete spectrum for the operator on the 
conditions that $k(s)\neq0$, for some $s \in \mathbb R$, and that $k(s)$ decays fast enough at infinity.
In addition, 
the authors assumed some regularity for $r$ and that $d-c$ is small enough.
In subsequent studies, the results were 
improved and generalized \cite{duclos, gold, davidkr}.
In \cite{gold}, the authors proved
the existence of discrete eigenvalues without the restriction on the width
of the strip.
In \cite{duclos} the authors estimated the number of discrete eigenvalues. 
They proved that this number is finite and bounded for a constant that does not depend on the width of $\Omega$.
%In particular, on of the results of \cite{dittrich} shows that if $r$ is a diffeomorfism of class $C^3$, 
%$k(s)$ has compact support and 
%$\int_{\mathbb R} k(s) \ds< 0$, then the discrete spectrum of $-\Delta_\Omega$ is nonempty.
In \cite{davidkr}, the authors proved the existence of discrete spectrum on the conditions 
that $r$ is of class $C^2$, $k(s)\neq0$, and $k(s) \to 0$, as $|s|\to \infty$;
we also emphasize that  
the authors 
made an overview of some new and old results on spectral properties
of $-\Delta_\Omega^D$, including other boundary conditions in $\partial \Omega$.
%r, let us mention a few other works involving broken waveguides
%such as

In  \cite{davidbriet}, the authors
introduced a new, two-dimensional model of strip to study the spectral problem of $-\Delta_\Omega^D$.
In that work,
$\Omega$ is a strip in $\mathbb R^2$ 
which is built by translating a segment oriented in a constant direction along an unbounded curve in the plane.
The spectrum of the operator $-\Delta_\Omega^D$ was carefully studied and the model 
covers different effects: purely essential spectrum, discrete spectrum or a combination of both.

One can consider strips embedded in a Riemannian manifold instead of the Euclidean space. For example,
suppose that $\Omega$ is a strip of constant width which is defined as a tubular neighborhood
of an infinite curve in a two-dimensional Riemannian manifold. This situation was considered in 
\cite{davidsurf}.
In particular, the author proved that the discrete spectrum of $-\Delta_\Omega^D$ 
is nonempty for non-negatively curved strips.

Strips on ruled surfaces are also regions of great interest \cite{davidhardy, davidruled, davidmainpaper, aleverri}; roughly speaking, 
a ruled surface is generated by straight lines translating along a curve in the
Euclidean space.
Consider the case where $\Omega$ is a
twisted strip composed of segments
translated along the straight line in $\mathbb R^3$ with respect to a rotation angle.
In \cite{davidruled}, assuming that the twisted effect diverges at infinity, the authors
studied the spectral problem of $-\Delta_\Omega^D$. In particular, they proved that this kind of geometry
can create discrete eigenvalues. In a similar situation,
in \cite{aleverri} was proved that
the discrete spectrum of $-\Delta_\Omega^D$ is nonempty since the twisted effect ``grows'' at infinity
while the width of $\Omega$ goes to zero.

Some results for strips on ruled surfaces can be extended to higher dimensions. In other words,
consider $\Omega$ as a twisted and bent two-dimensional strip embedded in $\mathbb R^d$ with $d \geq 3$.
In \cite{davidmainpaper}, the spectrum of the Dirichlet Laplacian $-\Delta_\Omega^D$ was carefully studied. The authors also proved that,
in the limit when the width of the strip tends to zero,
the Dirichlet Laplacian converges in the norm resolvent sense to a one-dimensional Schr\"odinger operator whose
potential depends on the deformations of twisting and bending.
An interesting point is that the geometric construction of those strips were performed with a 
{\it relatively parallel adapted frame} instead of a Frenet frame. In fact, it is known that 
a Frenet frame of a curve does not need to exist. However, the authors proved that a 
relatively parallel adapted frame always exists for an arbitrary curve. 
The goal of this work is to find additional 
information about the spectrum of the Dirichlet Laplacian in this situation.
In the next paragraphs, we present the formal construction of the
strip and more details of the problem.

Let $\Gamma: \mathbb R \to \mathbb R^{n+1}$, $n \geq 1$, be a curve of class $C^{1,1}$
parameterized by its arc-length $s$, i.e., $|\Gamma'(s)|=1$, for all $s \in \mathbb R$.
The vector $T(s):=\Gamma'(s)$ denotes its unitary tangent vector at the point  $\Gamma(s)$. Note that $T$ is a 
locally Lipschitz continuous function which
is differentiable almost everywhere in $\mathbb R$. The number $k(s):=|\Gamma''(s)|$, $s \in \mathbb R$,
is called the curvature of $\Gamma$ at the position $\Gamma (s)$.
In Appendix A of \cite{davidmainpaper}, the authors proved the existence of a relatively parallel adapted frame for 
the curve $\Gamma$. More exactly, it was shown that there exist $n$ almost-everywhere differentiable
normal vector fields $N_1, \cdots, N_n$ so that
\[
\left(
\begin{array}{c}
T \\
N_1 \\
\vdots \\
N_n
\end{array}\right)' 
=
\left(
\begin{array}{cccc}
0      & k_1    & \cdots & k_n \\
-k_1   & 0      & \cdots & 0   \\
\vdots & \vdots & \ddots & \vdots \\
-k_n   & 0      & \cdots &  0
\end{array}\right)
\left(
\begin{array}{c}
T \\
N_1 \\
\vdots \\
N_n
\end{array}\right),\]
where $k_j:\mathbb R \to \mathbb R$, $j \in \{1, \cdots, n\}$, are locally bounded functions.
In particular, the vector $(k_1, \cdots, k_n)$ satisfies
$k_1^2+ \cdots + k_n^2=k^2$.

Now, let $\Theta_j: \mathbb R \to \mathbb R$, $j \in \{1, \cdots, n\}$, be functions of class $C^{0,1}$ so that
\begin{equation}\label{condthetainnorm}
\Theta_1^2 + \cdots + \Theta_n^2 = 1,
\end{equation}
and define
\begin{equation}\label{normalfield}
N_\Theta: = \Theta_1 N_1 + \cdots + \Theta_n N_n.
\end{equation}

Consider the region
\begin{equation}\label{maindomain}
\Omega_\varepsilon := \{\Gamma(s) + N_\Theta(s) \, \varepsilon \, t: (s,t) \in \mathbb R \times (-1, 1)\};
\end{equation}
$\Omega_\varepsilon$ is obtained by translating the segment $(-\varepsilon, \varepsilon)$ along $\Gamma$ with respect to a normal field
(\ref{normalfield}).
Let $-\Delta_{\Omega_\varepsilon}^D$ be the Dirichlet Laplacian operator in $\Omega_\varepsilon$. More precisely,
$-\Delta_{\Omega_\varepsilon}^D$ is defined as the self-adjoint operator associated with the quadratic form
\begin{equation}\label{qfiqldef}
a_\varepsilon(\varphi) = \int_{\Omega_\varepsilon} |\nabla \varphi|^2 \dx, \quad 
\dom a_\varepsilon = H_0^1(\Omega_\varepsilon);
\end{equation}
$\nabla \varphi$ denotes the gradient of $\varphi$
corresponding to the metric induced by the immersion defined in (\ref{immeintr}), Section \ref{secaochange}.
For simplicity, we denote $-\Delta_\varepsilon : = -\Delta_{\Omega_\varepsilon}^D$.

Define the function $\Theta: \mathbb R \to \mathbb R^{n}$,
\begin{equation*}\label{defthetaincoord}
\Theta(s):= (\Theta_1(s), \cdots, \Theta_n(s)),
\end{equation*}
and write $|\Theta'(s)| := (\Theta_1^{'2}(s) + \cdots + \Theta_n^{'2}(s))^{1/2}$. 
Note that $\Theta \in C^{0,1}(\mathbb R; \mathbb R^{n})$.
As in \cite{davidmainpaper}, $\Theta$ is called {\it twisting vector};
if $\Theta'=0$, $\Omega_\varepsilon$ is called 
{\it untwisted} or {\it purely bent} strip; if $k \cdot \Theta := k_1 \Theta_1 + \cdots + k_n \Theta_n = 0$, $\Omega_\varepsilon$ is called 
{\it unbent} or {\it purely twisted} strip.
Geometrically, interpreting $\Gamma$ as a curve in $\Omega_\varepsilon$, 
$k \cdot \Theta$ is the {\it geodesic curvature} of $\Gamma$;
$-|\Theta'|^2/(1+|\Theta'|^2 \varepsilon^2 t^2)^2$ is the {\it Gauss curvature} of $\Omega_\varepsilon$.
One can see \cite{davidmainpaper} for a detailed geometric description of $\Omega_\varepsilon$.

Let $(\pi/2\varepsilon)^2$ be the first eigenvalue of the Dirichlet Laplacian 
$-\Delta_{(-\varepsilon,\varepsilon)}^D$ in $L^2(-\varepsilon,\varepsilon)$.
In \cite{davidmainpaper}, the authors presented a detailed study of the spectrum of $-\Delta_\varepsilon$. 
In particular,
under the conditions 
$k \cdot \Theta,  |\Theta'| \in L^\infty(\mathbb R)$, 
$\varepsilon \| k \cdot \Theta \|_{L^\infty(\mathbb R)} <1$,
and $(k \cdot \Theta)(s) \to 0$, $|\Theta'(s)| \to 0$, as $|s| \to \infty$,
they proved that: 

\vspace{0.2cm}
\noindent
(i) $\sigma_{ess}(-\Delta_\varepsilon) = [(\pi/2\varepsilon)^2, \infty)$;

%\vspace{0.2cm}
%\noindent
%In addition to the conditions in (i), 
%conditions

\vspace{0.2cm}
\noindent
(ii) if $\Theta'=0$ and $k \cdot \Theta \neq 0$, then the discrete spectrum of $-\Delta_\varepsilon$ is nonempty;

\vspace{0.2cm}
\noindent
(iii) if $k \cdot \Theta =0$ and $\Theta' \neq 0$ with $\varepsilon \|\Theta'\|_{L^\infty(\mathbb R)} \leq \sqrt{2}$,
then the discrete spectrum of $-\Delta_\varepsilon$ is empty.

\vspace{0.2cm}
\noindent

%The conditions in (i) and (ii) ensure that bend effect is sufficient to $-\Delta_\varepsilon$.
%create discrete eigenvalues for 

The conditions in (i) and (iii) show that a local twisted effect does not create discrete eigenvalues for 
$-\Delta_\varepsilon$.
Then, 
a natural question is to know if some appropriated twisted effect can be used to create discrete eigenvalues for the operator.
This problem
is the subject of this work.

In the first part of this paper, we assume that 
\begin{equation}\label{mainassuption}
\Theta \in C^{1,1}(\mathbb R; \mathbb R^n), \quad
(k \cdot \Theta)(s) = 0, \quad \hbox{and} \quad |\Theta'(s)| = \gamma - \beta(s), \quad
\forall s \in \mathbb R,
\end{equation}
where $\gamma$ is a positive number, and $\beta : \mathbb R \to \mathbb R$ is a continuous, almost-everywhere differentiable function with compact support so that $\beta' \in L^\infty(\mathbb R)$.
Geometrically, the second and third conditions in (\ref{mainassuption}) mean that $\Omega_\varepsilon$ is
a purely twisted strip; the properties of the function $\beta(s)$
imply that the twisted effect locally slows down. 
In this situation, we will find information about the spectrum of $-\Delta_\varepsilon$.
In particular, we will show that the conditions in (\ref{mainassuption}) 
can create discrete eigenvalues for $-\Delta_\varepsilon$. 
Note that the condition $k \cdot \Theta = 0$ does not necessarily  implies that $\Gamma$ is a straight line.

At first, we study the essential spectrum of $-\Delta_\varepsilon$. The strategy
is based on a
direct integral decomposition of the operator; see Section \ref{secessentspectrum}.
In particular, consider the one-dimensional operator
\begin{equation}\label{onedimopp0}
D_\varepsilon (0):= - \frac{\partial_t^2}{\varepsilon^2} + Y_\varepsilon^0(t), \quad \
\dom D_\varepsilon (0) = H_0^1(-1, 1) \cap H^2(-1, 1),
\end{equation}
where $\partial_t:=\partial/\partial t$,
\begin{equation}\label{fiberp0not}
Y_\varepsilon^0(t) := - \frac{3 \gamma^4 \varepsilon^2 t^2}{4 h_\varepsilon^4(t)} + 
\frac{\gamma^2}{2h_\varepsilon^2(t)}, \quad
h_\varepsilon(t):= \sqrt{1+ \gamma^2 \varepsilon^2 t^2}.
\end{equation}
Since $Y_\varepsilon^0 \in C^\infty[-1,1]$, $D_\varepsilon(0)$ has compact resolvent.
Denote by $\lambda_{\varepsilon,1}(0)$ its first eigenvalue and by $u_{\varepsilon, 1}^0$ the corresponding 
orthonormal eigenfunction; $\lambda_{\varepsilon, 1}(0)$ is simple.
Take $\varepsilon_0 > 0$ so that 
\begin{equation}\label{poseigenpot}
Y_\varepsilon^0 > 0, \quad \forall \varepsilon \in (0, \varepsilon_0).
\end{equation} 
Thus, for each $\varepsilon \in (0, \varepsilon_0)$,  $u_{\varepsilon, 1}^0$ can be chosen to be real and positive in $(-1,1)$;
see, e.g., Chapter 6 of \cite{evans} for more details.
The condition in (\ref{poseigenpot}) will be useful in the proof of Proposition \ref{proposition2} in Section \ref{secessentspectrum}.

Now, we have conditions to state the following result.
\begin{Theorem}\label{essentialspectheo}
Assume the conditions (\ref{condthetainnorm}), (\ref{mainassuption}) and (\ref{poseigenpot}). Then,
\[\sigma_{ess} (-\Delta_\varepsilon) = [\lambda_{\varepsilon, 1}(0), \infty).\]
\end{Theorem}
Theorem
\ref{essentialspectheo} is a consequence of Propositions \ref{proposition1} and \ref{proposition2} 
of  Section \ref{secessentspectrum}. 
In that same section,  Remark \ref{remarlasymplam} gives an asymptotic behavior for the sequence
$\{\lambda_{\varepsilon,1}(0)\}_\varepsilon$; in fact,
$\varepsilon^2 \lambda_{\varepsilon, 1}(0) \to (\pi/2)^2$, as $\varepsilon \to 0$.

The next results provide sufficient conditions to the existence of (discrete) eigenvalues for 
$-\Delta_\varepsilon$ below $\lambda_{\varepsilon, 1}(0)$.
Let $s_0 > 0$ be so that $\operatorname{supp} \beta \subset [-s_0, s_0]$.

\begin{Theorem}\label{intmaintheodisspec}
Assume the conditions (\ref{condthetainnorm}), (\ref{mainassuption}) and (\ref{poseigenpot}). If $\int_{-s_0}^{s_0} (|\Theta'(s)|^2 - \gamma^2) \ds < 0$, then there exists
$\varepsilon_1 > 0$ so that, for each $\varepsilon \in (0, \varepsilon_1)$,
\[\inf \sigma(-\Delta_\varepsilon) < \lambda_{\varepsilon, 1}(0),\]
i.e.,
$\sigma_{dis} (-\Delta_\varepsilon) \neq \emptyset$.
\end{Theorem}

The assumption $\int_{-s_0}^{s_0} (|\Theta'(s)|^2 - \gamma^2) \ds < 0$ in 
Theorem \ref{intmaintheodisspec}
is not a necessary condition to create discrete eigenvalues for $-\Delta_\varepsilon$. For example,

\begin{Theorem}\label{casetheodisspec}
Assume the conditions (\ref{condthetainnorm}), (\ref{mainassuption}) and (\ref{poseigenpot}). 
If $\int_{-s_0}^{s_0} (|\Theta'(s)|^2 - \gamma^2) \ds = 0$, then there exists
$\varepsilon_2 > 0$ so that, for each $\varepsilon \in (0, \varepsilon_2)$,
\[\inf \sigma(-\Delta_\varepsilon) < \lambda_{\varepsilon, 1}(0),\]
i.e.,
$\sigma_{dis} (-\Delta_\varepsilon) \neq \emptyset$.
\end{Theorem}

Theorems \ref{intmaintheodisspec} and \ref{casetheodisspec} show as an appropriated twisted effect can create
discrete eigenvalues for $-\Delta_\varepsilon$; its proofs are presented in Section \ref{sectiondiscspec}.

%\begin{Remark}{\rm
%Theorems \ref{intmaintheodisspec} and \ref{casetheodisspec} recall the following situation of twisted 
%waveguides.
%Let $S$ be a connected open and bounded subset of $\mathbb R^2$ and $\alpha(s)$ 
%a continuous function.
%Consider the twisted waveguide  $\omega :=\mathbb R \times \alpha(s) S$. 
%Assume that 
%the derivative $\alpha'(x)$ of the rotation angle can be written as 
%$\alpha'(s) = \rho_0 - \rho(s)$, where
%$\rho(s)$ is a bounded function with compact support in the interval $[-s_0, s_0]$.
%Consider the Dirichlet Laplacian operator in $\omega$. 
%This situation was considered in XX, where, in particular, the authors proved that
%either of the following
%conditions

%\vspace{0,2cm}
%\noindent
%(i)   $S$ is not rotationally symmetric, and $\int_{-s_0}^{s_0} (\alpha'(s)^2 - \rho_0^2) \ds < 0$;

%\vspace{0,2cm}
%\noindent
%(ii)  $S$ is not rotationally symmetric, $\alpha'(x) + \rho_0^2 > 0$, for all $|s| \leq s_0$, $\alpha''(s)$
%exists and belongs to $L^2[-s_0, s_0]$, and $\int_{-s_0}^{s_0} (\alpha'(s)^2 - \rho_0^2) \ds = 0$; 

%\vspace{0,2cm}
%\noindent
%implies the existence of discrete spectrum for the operator.

%$\beta$ is a positive number and $\rho(x) \approx L |x|^{-\nu}$, as $|x| \to \infty$.
%This situation was studied in XX. On the conditions that

%}
%\end{Remark}

Now, 
the goal of the second part of this paper is to discuss a different situation where this phenomenon can also be obtained.

As already commented in the previous paragraphs, in \cite{davidruled, aleverri}, the authors
studied the Dirichlet Laplacian restricted in a two-dimensional twisted (straight) strip in $\mathbb R^3$.
We emphasize the paper \cite{davidruled} where the authors considered the region
\begin{equation}\label{regiondavidruled}
{\cal S}:=\{(s, t\cos \theta(s),  t\sin \theta(s)): (s,t) \in \mathbb{R}\times(a_1, a_2)\},
\end{equation}
where $\theta: \mathbb R \to \mathbb R$ is a (locally) Lipschitz continuous function and 
$a_2, a_2 \in \mathbb R$.
One of the results of that work shows that the 
assumptions
\begin{equation}\label{davidruledeq}
\lim_{|s|\to \infty} |\theta'(s)| = \infty, \quad a_1a_2 \leq 0,
\end{equation}
ensure that the discrete spectrum of the Dirichlet Laplacian in ${\cal S}$ is nonempty.
As noted by the authors, an interesting geometric fact is that the conditions in (\ref{davidruledeq}) show that the two-dimensional strip ${\cal S}$
looks at infinity like a three-dimensional tube of annular
cross-section $\{x \in \mathbb R^2: 0 < |x| < r\}$, 
where $r:=\max\{|a_1|, |a_2|\}$. 
Inspired by \cite{davidruled}, in \cite{aleverri} the author
considered a twisted strip similar to that in (\ref{regiondavidruled}). However, 
in that work,
the twisted effect ``grows'' at infinity while the width of the strip goes to zero. Then, since
the strip is thin enough, it was shown that the discrete spectrum of the Dirichlet Laplacian is nonempty and
was found an asymptotic behavior for the eigenvalues.
In the next paragraphs, we present an adaptation the results of \cite{aleverri} 
for the model of strips treated in this work.

In this new situation, for $n \geq 1$, assume that $\Theta : \mathbb R \to \mathbb R^n$ is a
$C^{1,1}$ function, which satisfies (\ref{condthetainnorm}), and

\vspace{0.2cm}
\noindent
(I) $\displaystyle \lim_{|s| \to \infty}|\Theta'(s)| = \infty$;

\vspace{0.2cm}
\noindent
(II)
$|\Theta'|$ is decreasing in $(-\infty, 0)$ and increasing in $(0, \infty)$.

\vspace{0.2cm}

Fix a number $0 < a < 1/3$.
For each $\varepsilon > 0$ small enough, let $\nu_1(\varepsilon) < 0$ and
$\nu_2(\varepsilon) > 0$ so that
\begin{equation}\label{conthetdernu12}
|\Theta'(\nu_i(\varepsilon))| = \frac{1}{\varepsilon^a}, \quad i \in \{1,2\}.
\end{equation}
Define $I_\varepsilon := (\nu_1(\varepsilon), \nu_2(\varepsilon))$ and let
$\Theta_\varepsilon : \mathbb R \to \mathbb R^{n}$ be a function of class 
$C^{1,1}$ so that

\vspace{0.2cm}
\noindent
(III) $\Theta_\varepsilon(s) = \Theta(s)$,  for all $s \in I_\varepsilon$;

\vspace{0.2cm}
\noindent
(IV) $|\Theta_\varepsilon'(s)| \leq  |\Theta'(s)|$, for all $s \in \mathbb R$;

\vspace{0.2cm}
\noindent
(V) $|\Theta_\varepsilon'(s)|$ is non-increasing in $(-\infty, 0)$ and  non-decreasing in $(0, +\infty)$.

%\vspace{0.2cm}
%\noindent
%(V) $q(\varepsilon) := \displaystyle\liminf_{|s| \to \infty} |\Theta_\varepsilon'(s)| \to \infty$, as $\varepsilon \to 0$;

\vspace{0.2cm}
\noindent
Now, assume that the sequence $\{\Theta_\varepsilon\}_\varepsilon$ satisfies 

\vspace{0.2cm}
\noindent
(VI) there exists $K> 0$ so that
\[|\Theta_\varepsilon'(s)| \leq \frac{K}{\varepsilon^a}, 
\quad |\Theta_\varepsilon''(s)| \leq \frac{K}{\varepsilon^b}, \quad 
|\Theta_\varepsilon'''(s)| \leq \frac{K}{\varepsilon^c}, \quad \forall s \in \mathbb R,\]
for all $\varepsilon > 0$ small enough, where $b,c$ are real numbers so that $b< 1,$ $a+c < 2$;

\vspace{0.2cm}
\noindent
(VII) $|\Theta_{\varepsilon}'(s)| \leq |\Theta_{\varepsilon'}'(s)|$, for all $s \in \mathbb R$,
if $\varepsilon > \varepsilon'$.

\vspace{0.2cm}
\noindent
Finally, 
we use the notation $\Theta_\varepsilon := (\Theta_1^\varepsilon, \cdots, \Theta_n^\varepsilon)$,
and
we assume that

\vspace{0.2cm}
\noindent
(VIII) $(\Theta_1^\varepsilon)^2 + \cdots + (\Theta_n^\varepsilon)^2 = 1$.

\vspace{0.2cm}

For each $\varepsilon > 0$ small enough, let $\Gamma_\varepsilon : \mathbb R \to \mathbb R^{n+1}$ 
be a curve of class $C^{1,1}$  whose curvature
$k_\varepsilon$ satisfies

\vspace{0.2cm}
\noindent
(IX) $\operatorname{supp} k_\varepsilon \subset I_\varepsilon$, 
and $(k_\varepsilon \cdot \Theta_\varepsilon)(s) =0$, for all $s \in \mathbb R$.

\vspace{0.2cm}
\noindent
%As a consequence, $(k_\varepsilon \cdot \Theta_\varepsilon)(s) =0$, for all $s \in \mathbb R$.
The normal vector fields of $\Gamma_\varepsilon$ are denoted by
$N_1^\varepsilon, \cdots, N_n^\varepsilon$, and
$N_{\Theta_\varepsilon}^\varepsilon:= \Theta_1^\varepsilon N_1^\varepsilon + \cdots + \Theta_n^\varepsilon N_n^\varepsilon$.

Consider the strip
\[\widetilde{\Omega}_\varepsilon : = 
\{\Gamma_\varepsilon(s) + N_{\Theta_\varepsilon}^\varepsilon(s) \, \varepsilon \, t: (s,t) \in \mathbb R \times (-1, 1)\}.\]
Geometrically,  $\widetilde{\Omega}_\varepsilon$ is a locally twisting (and locally bending) strip for which the
twisted effect  ``grows'' at infinity while its width 
goes to zero.

Let $-\Delta_{\widetilde{\Omega}_\varepsilon}^D$ be the Dirichlet Laplacian operator in $\widetilde{\Omega}_\varepsilon$, i.e., 
the self-adjoint operator associated with the quadratic form
\begin{equation}\label{atildintsecd}
\tilde{a}_\varepsilon(\varphi) =\int_{\widetilde{\Omega}_\varepsilon} |\nabla \varphi|^2 \dx, \quad 
\dom \tilde{a}_\varepsilon = H_0^1(\widetilde{\Omega}_\varepsilon).
\end{equation}
For simplicity, write $-\tilde{\Delta}_\varepsilon := -\Delta_{\widetilde{\Omega}_\varepsilon}^D$.

\begin{Remark}
{\rm 
Let $T$ be a self-adjoint operator that
is bounded from below. We denote by  $\{\lambda_j(T)\}_{j \in \mathbb N}$ the non-decreasing sequence of numbers corresponding to the
spectral problem of $T$ according to the Min-Max Principle; see, for example, Theorem XIII.1 in \cite{reed}.}
\end{Remark}

Let $\Pi(\varepsilon)$ be the
infimum of the essential spectrum  of $-\tilde{\Delta}_\varepsilon $.
Denote by $N(\varepsilon) \leq \infty$ the number of eigenvalues $\lambda_j(-\tilde{\Delta}_\varepsilon )$
of $-\tilde{\Delta}_\varepsilon $ below $\Pi(\varepsilon)$. 
Let $-\Delta_\mathbb R$ be the one-dimensional Laplacian and 
consider the self-adjoint operator $-\Delta_\mathbb R + (|\Theta'(s)|^2/2) \Id$ acting in
$L^2(\mathbb R)$; $\Id$ denotes the identity operator in $L^2(\mathbb R)$.
Due to the condition (I), this operator
has purely discrete spectrum. 
In Section \ref{sectionthinstrips} of this work, we prove the following result.

\begin{Theorem}\label{theothinstripdiv}
Assume the conditions (I)-(IX).
For $\varepsilon > 0$ small enough, the discrete spectrum $\sigma_{dis}(-\tilde{\Delta}_\varepsilon )$ is nonempty
and $N(\varepsilon) \to \infty$, as $\varepsilon \to 0$. Furthermore, for each $j \in \mathbb N$,
\[\lim_{\varepsilon \to 0} \left[\lambda_j(-\tilde{\Delta}_\varepsilon) - \left(\frac{\pi}{2 \varepsilon}\right)^2 \right] =
\lambda_j\left(-\Delta_\mathbb R  + \frac{|\Theta'(s)|^2}{2} \Id \right).\]
\end{Theorem}

Estimating the number of discrete eigenvalues of the Dirichlet Laplacian operator in unbounded strips is also an interesting problem \cite{daugela, daugeray, daugebonafos, aleverri}. 
In the case of curved strips in $\mathbb R^2$, 
in  \cite{duclos}, 
the authors found that this number is
finite and bounded for a constant that does not depend on the width of the strip.
Now, consider the ``broken'' strip defined by 
$\{(s,t) \in \mathbb R^2: s \tan \theta < |t| < (s+\pi / \sin \theta) \tan \theta\}$, $\theta \in (0, \pi/2)$.
In \cite{daugela} , the authors ensured the existence of discrete eigenvalues
for the Dirichlet Laplacian operator in that region. They proved that
there are a finite number of them and this number tends to infinity as $\theta \to 0$.

As a consequence of  Theorem \ref{theothinstripdiv},
the number of discrete eigenvalues of
$-\tilde{\Delta}_\varepsilon$ grows when the width of the strip goes to zero. 
In \cite{aleverri}, a result similar to Theorem \ref{theothinstripdiv} was obtained as a consequence
of a convergence in the norm resolvent sense of the operators associated to the problem. 
In this text, we present a
simpler proof where the strategy is based on to find upper and lower bounds for the eigenvalues 
$\lambda_j(-\tilde{\Delta}_\varepsilon)$.

%This phenomenon was also noted in \cite{aleverri}.

We finish this introduction with some remarks and examples of the model presented.

\begin{Remark}{\rm
Theorem \ref{theothinstripdiv} shows that the locally bending effect imposed by 
(IX) does not affect the final
result.}
\end{Remark}

\begin{Example}{\rm
A simple example of a family of curves $\{\Gamma_\varepsilon\}_\varepsilon$ satisfying (IX) is the following.
Let $\Theta$ be a function satisfying (I), and
let $\Gamma : \mathbb R \to \mathbb R^{n+1}$ be a curve of class 
$C^{2,1}$ whose curvature $k$ has compact support and satisfies $(k \cdot \Theta)(s)=0$, for all $s \in \mathbb R$. 
Define $\Gamma_\varepsilon := \Gamma$, for all $\varepsilon > 0$ small enough.
Then, the condition 
(IX) is satisfied for all $\varepsilon > 0$ small enough.}
\end{Example}

\begin{Remark}{\rm
In the case $n=2$, conditions 
(\ref{condthetainnorm}) and (I) imply that $\Theta$ can be written as 
$\Theta = (\cos(\psi), \sin(\psi))$, for some function $\psi \in C^{1,1}(\mathbb R;\mathbb R)$ so that
$|\psi'(s)| \to \infty$, as $|s|\to\infty$.
In the case $n=3$, due to the condition (\ref{condthetainnorm}),
$\Theta$ 
can be written as 
$\Theta = (\cos(\phi)\cos(\psi), \sin(\phi)\cos(\psi), \sin(\psi))$, for functions $\phi, \psi \in C^{1,1}(\mathbb R; \mathbb R)$. 
Since $|\Theta'| = \sqrt{(\phi')^2  \cos^2(\psi) + (\psi')^2}$, the condition  (I) is  satisfied, 
for example, if
$|\psi'(s)| \to \infty$, as $|s|\to\infty$,
or, if $\psi(s) \in (c_1, c_2] \subset (0, \pi/2)$, for all $s\in \mathbb R$, and 	$|\phi'(s)| \to \infty$, as $|s|\to\infty$.
	}
%% OBS: DERIVADA $$\frac{2 s (2 - \cos(s^2) \sin(s^2))}{\sqrt{4 s^2 + \cos^2(s^2)}}$$
\end{Remark}

\begin{Example}\label{examplen2}
{\rm
Consider  $\Theta: \mathbb R \to \mathbb R^2$ defined by $\Theta(s)= (\cos(s^2), \sin(s^2))$.
Some calculations show that $|\Theta'(s)|= 2|s|$, $s \in \mathbb R$; $\Theta$ satisfies (I)-(II).
Fix a number $0 < a < 1/3$. 
For each $\varepsilon > 0$ small enough,
define the functions  $\alpha_\varepsilon: \mathbb R \longrightarrow \mathbb R$,
\[\alpha_\varepsilon(s) := \left\{\begin{array}{ll}
(-2\varepsilon^a s -1)/\varepsilon^{2a},  &   s \leq -1/\varepsilon^a,\\
s^2, &    s\in (-1/\varepsilon^a, 1/\varepsilon^a),\\
(2\varepsilon^a s -1)/\varepsilon^{2a},  &  s\geq 1/\varepsilon^a,
\end{array}\right.\]
and
$\Theta_\varepsilon(s) :=  (\cos(\alpha_\varepsilon(s)), \sin(\alpha_\varepsilon(s)))$, $s\in \mathbb R$.
Taking $b=2a$ and $c=3a$, the conditions in (VI) are satisfied.
In fact, the sequence $\{\Theta_\varepsilon\}_\varepsilon$ satisfies the conditions (III)-(VIII).
	}
\end{Example}

\begin{Example}{\rm
Consider the function  ${\Theta}: \mathbb R \to \mathbb R^3$ defined by
$$\Theta(s) = \left(\cos(s^2)\cos\left(\frac{1}{1+s^2}\right), \sin(s^2)\cos\left(\frac{1}{1+s^2}\right), \sin\left(\frac{1}{1+s^2}\right)\right).$$
Note that the condition 
(\ref{condthetainnorm}) is satisfied and 
the function
$$|\Theta'(s)|= \sqrt{4s^2 \cos^2\left(\frac{1}{1+s^2}\right) + \frac{4s^2}{(1+s^2)^4}}, \quad s \in \mathbb R,$$
satisfies (I)-(II).
Fix a number $0 < a < 1/3$.
For each $\varepsilon>0$ small enough, 
take $\nu(\varepsilon)> 0$ so that
$|\Theta'(\pm \nu(\varepsilon))|
=1/\varepsilon^a$; then, there exists $K>0$ so that $|\nu(\varepsilon)| \leq K /\varepsilon^a$, for all
$\varepsilon > 0$ small enough.
Let $\alpha_\varepsilon: \mathbb R \longrightarrow \mathbb R$ be the function defined by 
\[\alpha_\varepsilon(s) := \left\{\begin{array}{ll}
	-2\nu(\varepsilon)s -\nu(\varepsilon)^2,  &   s \leq - \nu(\varepsilon),\\
	s^2, &    s\in (-\nu(\varepsilon), \nu(\varepsilon)),\\
	2\nu(\varepsilon) s -\nu(\varepsilon)^2,  &  s\geq \nu(\varepsilon);
\end{array}\right.\]
note that $\alpha_\varepsilon \in C^{1,1}(\mathbb R; \mathbb R)$.
Now, define 
\[\Theta_\varepsilon(s) := \left(\cos(\alpha_\varepsilon(s))\cos\left(\frac{1}{1+s^2}\right), \sin(\alpha_\varepsilon(s))\cos\left(\frac{1}{1+s^2}\right), \sin\left(\frac{1}{1+s^2}\right)\right),\quad s \in \mathbb R.\]
The sequence $\{\Theta_\varepsilon\}_\varepsilon$ satisfies  the conditions (III)-(VII), where $b=2a$ and $c=3a$.
}\end{Example}

This paper is organized as follows. In Section \ref{secaochange} we present some details of the construction 
of the region in (\ref{maindomain}) and we make usual change of coordinates 
in the quadratic form in (\ref{qfiqldef}).
Sections \ref{secessentspectrum} and \ref{sectiondiscspec} are dedicated to study the essential and discrete
spectrum of $-\Delta_\varepsilon$, respectively.
In Section \ref{sectionthinstrips}, we study the spectral problem of $-\tilde{\Delta}_\varepsilon$.
In Appendix \ref{appendix001} are presented results that are useful in this text.
Along the text, $K$ is used to denote different constants.

\section{Geometry of the region and change of coordinates}\label{secaochange}

Recall the region $\Omega_\varepsilon$ given by (\ref{maindomain}) in the Introduction 
and the straight strip $\Lambda := \mathbb R \times (-1,1)$.
In this section we identify  $\Omega_\varepsilon$ with the 
Riemannian manifold $(\Lambda, {\cal G}_\varepsilon)$, where ${\cal G}_\varepsilon$ is given by (\ref{geodgaussdef}),
below. After that, we perform usual changes of coordinates in the quadratic form $a_\varepsilon(\varphi)$.

Consider the map
\begin{equation}\label{immeintr}
\begin{array}{rcll}
{\cal L}_\varepsilon: & \mathbb R^2 & \longrightarrow        &  \mathbb R^{n+1}, \\
& (s,t)       & \longmapsto    & \Gamma(s) +  N_\Theta(s) \varepsilon t .
\end{array}
\end{equation}
We have $\Omega_\varepsilon = \mathcal{L}_\varepsilon(\Lambda)$. 
Define the metric ${\cal G}_\varepsilon:= \nabla {\cal L}_\varepsilon \cdot (\nabla {\cal L}_\varepsilon)^\perp$.
Some calculations show that 
\begin{equation}\label{geodgaussdef}
{\cal G}_\varepsilon=\left(\begin{array}{cc}
f_\varepsilon^2  & 0 \\
0  & \varepsilon^2
\end{array}\right), \quad
f_\varepsilon(s,t):= \sqrt{1 +  |\Theta'(s)|^2 \varepsilon^2 t^2 }.
\end{equation}
Let ${\cal J}_\varepsilon$ be the Jacobian matrix of ${\cal L}_\varepsilon$. One has
$\det {\cal J}_\varepsilon = | \det {\cal G}_\varepsilon|^{1/2} = \varepsilon f_\varepsilon > 0$, for all $(s, t) \in \Lambda$.
If $\Gamma$ and $\Theta$ are smooth functions, the
map ${\cal L}_\varepsilon: \Lambda \to \Omega_\varepsilon$ is a local smooth diffeomorphism. 
Then, $\Omega_\varepsilon$ can be identified with the Riemannian manifold $(\Lambda, {\cal G}_\varepsilon)$.
However, as mentioned in the Introduction of this work, the assumptions about $\Gamma$ and $\Theta$ are more general.
At first, one has

%%%%%%%%%%%%%%%%%%%%%%%%%%%%%%%%%%%%%%%%%%%%%%%%%%%%%%%%%% 

\begin{Proposition}\label{propdiff}
Assume that $\Gamma \in C^{1,1}(\mathbb R; \mathbb R^{n+1})$ and $\Theta \in C^{0,1}(\mathbb R; \mathbb R^{n})$.
Then, the map  ${\cal L}_\varepsilon: \Lambda \longrightarrow \Omega_\varepsilon$ is 
a	local $C^{0,1}$-diffeomorphism.
\end{Proposition}

%%%%%%%%%%%%%%%%%%%%%%%%%%%%%%%%%%%%%%%%%%%%%%%%%%%%%%%%%%

The proof of this result can be found in \cite{davidmainpaper}.
We  emphasize that in that work does not necessarily $k \cdot \Theta = 0$, and
Proposition \ref{propdiff} is proven under the additional assumptions that
$k \cdot \Theta \in L^\infty(\mathbb R)$ and $\varepsilon \| k \cdot \Theta \|_{L^\infty(\mathbb R)} < 1$.

In particular,  Proposition \ref{propdiff} ensures that ${\cal L}_\varepsilon$ is a $C^{0,1}$-immersion.
In addition, assume that ${\cal L}_\varepsilon$ is injective. Thus, 
the strip $\Omega_\varepsilon$  does not self-intersect and it is interpreted as an immersed submanifold
in $\mathbb R^{n+1}$.
As a consequence, $(\Lambda, {\cal G}_\varepsilon)$ is an abstract Riemannian manifold.

Now, we perform 
a change of coordinates 
so that the quadratic form $a_\varepsilon(\varphi)$ starts to act in the 
Hilbert space $L^2(\Lambda)$ (with the usual metric of $\mathbb R^2$) instead of $L^2(\Omega_\varepsilon)$.
At first, consider the unitary operator
\[\begin{array}{rccc}
{\cal U}_\varepsilon: & L^2(\Omega_\varepsilon) & \longrightarrow        &  L^2( \Lambda, f_\varepsilon \ds\dt), \\
& \psi       & \longmapsto    & \varepsilon ^{1/2} \psi \circ {\cal L}_\varepsilon,
\end{array}\]
and define the quadratic form
\begin{align*}
b_\varepsilon (\psi)  
:= & \,
a_\varepsilon \left({\cal U}_\varepsilon^{-1} \psi \right) 
= \int_\Lambda \langle \nabla \psi , {\cal G}_\varepsilon^{-1} \nabla \psi \rangle f_\varepsilon \ds \dt \\	               
= & 
\int_\Lambda \frac{|\partial_s \psi |^2}{ f_\varepsilon } \ds \dt 
+  \frac{1}{\varepsilon^2}\int_\Lambda | \partial_t \psi  |^2 f_\varepsilon \ds\dt,
\end{align*}
$\dom b_\varepsilon := {\cal U}_\varepsilon (H^1_0 (\Omega_\varepsilon))$;
$\partial_t =\partial/\partial_t$ and $\partial_s =\partial/\partial_s$.
Then, consider 
\[\begin{array}{rccc}
{\cal V}_\varepsilon: & L^2(\Lambda) & \longrightarrow        &  L^2( \Lambda, f_\varepsilon \ds \dt), \\
& \psi       & \longmapsto    & f_\varepsilon^{-1/2} \psi,
\end{array}\]
which is also a unitary operator,
and, finally, define
\begin{equation*}\label{formab} 
c_\varepsilon(\psi)  := b_\varepsilon \left({\cal V}_\varepsilon \psi \right) =
\int_\Lambda \frac{1}{ f_\varepsilon^2 } \left|\partial_s \psi - 
\frac{ \partial_s f_\varepsilon}{2f_\varepsilon} \psi\right|^2\ds \dt   	+ 
\frac{1}{\varepsilon^2} \int_\Lambda |\partial_t \psi |^2 \ds\dt  +
\int_\Lambda V_\varepsilon |\psi|^2 \ds\dt,
\end{equation*}
where
\begin{equation*}\label{Vbeta}
V_\varepsilon(s,t) := 
-\frac{3|\Theta'(s)|^4 \varepsilon^2 t^2 }{4f_\varepsilon^4(s,t)}  
+\frac{|\Theta'(s)|^2}{2f_\varepsilon^2(s,t)},
\end{equation*}
$\dom c_\varepsilon = {\cal V}_\varepsilon^{-1}({\cal U}_\varepsilon (H^1_0 (\Omega_\varepsilon)))$.
Due to the conditions in (\ref{mainassuption}),
one has  $\dom c_\varepsilon = H_0^1(\Lambda)$.
Denote by $C_\varepsilon$ the self-adjoint operator associated with the
quadratic form $c_\varepsilon(\psi).$

\section{Essential spectrum}\label{secessentspectrum}

This section is dedicated to prove Theorem \ref{essentialspectheo}.
Recall the functions  $h_\varepsilon$ and $Y_\varepsilon^0$ defined by
(\ref{fiberp0not}) in the Introduction.
Consider the quadratic form
\[d_\varepsilon(\psi) := \int_\Lambda \frac{|\partial_s \psi|^2}{h_\varepsilon^2 }\ds \dt 	+
\frac{1}{\varepsilon^2} \int_\Lambda |\partial_t \psi |^2 \ds\dt  	+
\int_\Lambda Y_\varepsilon^0 |\psi|^2 \ds\dt, 	\quad  	\dom d_\varepsilon := H^1_0(\Lambda).\]
Denote by $D_\varepsilon$ the self-adjoint operator associated with $d_\varepsilon(\psi)$.
We start with the following result.

\begin{Proposition}\label{proposition1}
Assume the conditions in (\ref{mainassuption}).
Then,	$\sigma_{ess}(C_\varepsilon) = \sigma_{ess}(D_\varepsilon).$
\end{Proposition}

The proof of this result is presented in Appendix \ref{appendix001}. 
As a consequence, we start to study the essential spectrum of $D_\varepsilon$.

Let ${\cal F}_s : L^2(\Lambda) \to L^2(\Lambda)$ be the Fourier transform with respect to $s$.
${\cal F}_s$ is a unitary operator and, for functions in $L^1(\Lambda)$, its action is given by
\begin{equation*}
	  ( {\cal F}_s \psi )(p,t) = 
	  \frac{1}{\sqrt{2\pi}} \int_\mathbb R e^{-i ps} \psi(s,t)\ds.
\end{equation*}
Then, the operator $\hat{D}_\varepsilon :=  {\cal F}_s D_\varepsilon {\cal F}_s^{-1}$ admits  the direct integral decomposition
\begin{equation}\label{decfou}
\hat{D}_\varepsilon = \int_\mathbb{R}^\oplus D_\varepsilon(p)\, \mathrm{d} p,
\end{equation}
where, for each $p \in \mathbb R$, $D_\varepsilon(p)$ is the self-adjoint operator associated with the
quadratic form
\[d_\varepsilon(p)(v) := \frac{1}{\varepsilon^2}\int_{-1}^{1} |\partial_t v |^2 \dt 
+ \int_{-1}^{1} Y_\varepsilon^p |v|^2 \dt,  \quad \dom d_\varepsilon(p) = H_0^1(-1,1), \]
where $Y_\varepsilon^p(t):= p^2/h_\varepsilon^2(t) + Y_\varepsilon^0(t)$. 
More precisely, 
\[D_\varepsilon(p) = - \frac{\partial_t^2}{\varepsilon^2} + Y_\varepsilon^p(t), \quad 
\dom D_\varepsilon(p) = H^2(-1,1)\cap H^1_0(-1,1);\]
the case $p=0$ corresponds to the operator defined by (\ref{onedimopp0}) in the Introduction.
Since $Y_\varepsilon^p \in C^\infty[-1,1]$, each $D_\varepsilon(p)$ has compact resolvent.
Denote by $\{\lambda_{\varepsilon, n}(p)\}_{n \in \mathbb N}$ the sequence of eigenvalues of $D_\varepsilon(p)$ and by
$\{u_{\varepsilon,n}(p)\}_{n \in \mathbb N}$ the sequence of the corresponding normalized eigenfunctions, i.e.,
\[D_\varepsilon(p) u_{\varepsilon,n}(p) = \lambda_{\varepsilon,n}(p) u_{\varepsilon,n}(p), \quad
n\in \mathbb N, \quad p\in\mathbb R.\]
Due to  the decomposition in (\ref{decfou}), we have
\begin{equation}\label{uniaoespectral}
\sigma (D_\varepsilon) = \cup_{p \in \mathbb R} \sigma( D_\varepsilon(p)) 
= \cup_{n\in\mathbb N}\left\{\lambda_{\varepsilon,n}(p) : p \in \mathbb R \right\}.
\end{equation}
In particular, denote $u_{\varepsilon,1}^0:=u_{\varepsilon,1}(0)$. Due to the condition (\ref{poseigenpot}) in the Introduction, $u_{\varepsilon,1}^0$ can be chosen to be real and positive in $(-1,1)$. This property will be used in the proof of Proposition \ref{proposition2} below.

\begin{Lemma}\label{limiteautovalores}
For each $n\in\mathbb N,$ $\lambda_{\varepsilon,n}(\cdot)$ is a real analytic function
in $p$ and 
\begin{equation*}
	\lim_{p\to \pm \infty} \lambda_{\varepsilon,n}(p) = \infty.
\end{equation*}
\end{Lemma} 
\begin{proof} 
At first, note that $\dom D_\varepsilon(p) = \dom D_\varepsilon(0)$, for all $p \in \mathbb R$.
One can write $D_\varepsilon(p) = D_\varepsilon(0)+ p^2/h_\varepsilon^2$.	
Since $h_\varepsilon \geq 1$, it holds the estimate
$\|(p^2/h_\varepsilon^2) v \| \leq p^2 \|v\|$,
for all $v \in \dom D_\varepsilon(0)$,    for all $p \in \mathbb R$.
Then, $p^2/h_\varepsilon^2$  is $D_\varepsilon(0)$-bounded with relative bound zero.
Consequently, $\{D_\varepsilon(p): p \in \mathbb R\}$ is a type A analytic family.
By Theorem 3.9 in \cite{kato}, $\lambda_{\varepsilon,n}(\cdot)$ is a real analytic function in $p$.
		
Now, for each $v \in \dom d_\varepsilon(p)$, we have
\[d_\varepsilon(p)(v) = \frac{1}{\varepsilon^2}\int_{-1}^{1} |\partial_t v |^2 \dt 
+ \int_{-1}^{1} Y_\varepsilon^p |v|^2\dt
\geq 
\left(\frac{p^2}{1+\gamma^2 \varepsilon^2} - \varepsilon^2 \gamma^4 \right) \int_{-1}^1 |v|^2 \dt.\]
As a consequence, for each $n \in \mathbb N$, 
\[\lambda_{\varepsilon,n}(p) = d_{\varepsilon}(p) (u_{\varepsilon,n}(p))
\geq 
\frac{p^2}{1+\gamma^2 \varepsilon^2} - \varepsilon^2\gamma^4, 
\quad p \in \mathbb R.\]
Thus, we obtain the punctual limit $\lambda_{\varepsilon, n}(p) \to \infty$, as  $p \to \pm \infty.$ 
\end{proof}

\begin{Proposition}\label{proposition2}
One has $\sigma (D_\varepsilon) = [\lambda_{\varepsilon,1}(0),\infty).$
\end{Proposition}
\begin{proof}
Since $\lambda_{\varepsilon, 1}(p)$ is a real analytic function in $p$, by (\ref{uniaoespectral}), we have
$[\lambda_{\varepsilon,1}(0),\infty) \subset  \sigma(D_\varepsilon).$
Now, we need to show that 
\begin{equation}\label{intersecao}
(-\infty, \lambda_{\varepsilon,1}(0)) \cap \sigma(D_\varepsilon) =  \emptyset.
\end{equation} 
Take $\psi \in C_0^\infty(\Lambda)$.
Since $u_{\varepsilon, 1}^0$ is positive, 
we can write $\psi(s,t) = \phi(s,t) u_{\varepsilon, 1}^0(t)$, with $\phi \in C_0^\infty(\Lambda)$.
Some calculations show that 
\begin{align*}
d_\varepsilon(\psi) -\lambda_{\varepsilon,1}(0) \int_\Lambda |\psi|^2 \ds \dt
= & 
\int_\Lambda \left(\frac{|\partial_s\phi|^2}{h_\varepsilon^2}
+ \frac{|\partial_t \phi|^2}{\varepsilon^2} \right) |u_{\varepsilon,1}^0|^2 \ds \dt 
+ \frac{2}{\varepsilon^2}\, {\cal R} \int_\Lambda  \overline{\phi}\, \partial_t\phi \, u_{\varepsilon,1}^0 
\partial_t u_{\varepsilon,1}^0 \ds \dt \\
& +  
\frac{1}{\varepsilon^2} \int_\Lambda |\phi|^2 |\partial_t u_{\varepsilon,1}^0|^2 \ds \dt
+ \int_\Lambda |\phi|^2 (Y_\varepsilon^0 |u_{\varepsilon,1}^0|^2 
- \lambda_{\varepsilon,1}(0) |u_{\varepsilon,1}^0|^2) \ds \dt.
\end{align*}
An integration by parts, and since 
$D_\varepsilon(0) u_{\varepsilon,1}^0 = \lambda_{\varepsilon,1}(0) u_{\varepsilon,1}^0$, one has
\[d_\varepsilon(\psi) - \lambda_{\varepsilon,1}(0) \int_\Lambda |\psi|^2 \ds \dt  =  
\int_\Lambda\left(\frac{|\partial_s\phi|^2}{h_\varepsilon^2} 
+ \frac{|\partial_t \phi|^2}{\varepsilon^2}\right) |u_{\varepsilon,1}^0|^2 \ds \dt
\geq 0.\]
Then, we obtain (\ref{intersecao}).
\end{proof}

\begin{proof}[\bf Proof of Theorem \ref{essentialspectheo}]
Apply Propositions  \ref{proposition1} and \ref{proposition2}.	
\end{proof}

\begin{Remark}\label{remarlasymplam}
{\rm
For the sequence $\{\lambda_{\varepsilon,1}(0)\}_\varepsilon$, we have the estimate:
there exists $K > 0$ so that
\[\int_{-1}^{1} |\partial_t v|^2 \dt 
\leq
\varepsilon^2 d_\varepsilon(0)(v)=
\int_{-1}^{1} \left(|\partial_t v|^2 + \varepsilon^2 Y_\varepsilon^0  |v|^2 \right) \dt
\leq 
\int_{-1}^{1} |\partial_t v|^2 \dt + K \varepsilon^2  \int_{-1}^{1} |v|^2 \dt,
\]
for all $v \in H_0^1(-1,1)$, for all $\varepsilon > 0$ small enough.
Consequently,
\[\left(\frac{\pi}{2}\right)^2 \leq \varepsilon^2 \lambda_{\varepsilon, 1}(0) \leq  
\left(\frac{\pi}{2}\right)^2 + O(\varepsilon^2).\]
}
\end{Remark}

\section{Discrete spectrum}\label{sectiondiscspec}

Based in \cite{gold}, the main idea in the proofs of Theorems \ref{intmaintheodisspec} and \ref{casetheodisspec} is to find
a function $\psi \in \dom c_\varepsilon$ so that
$(c_\varepsilon(\psi)-\lambda_{\varepsilon,1}(0)\|\psi\|^2)/\|\psi\|^2 < 0.$

\begin{proof}[\bf Proof of Theorem \ref{intmaintheodisspec}]
Take $\delta>0$. Define $\psi_\delta(s,t) := \phi(s) u_{\varepsilon,1}^0(t)$, where
\[\phi(s) :=\left\{\begin{array}{ll}
e^{\delta (s+s_0)},  &   s\leq -s_0,\\
1, &    -s_0\leq s \leq s_0,\\
e^{-\delta (s-s_0)},  &  s\geq s_0.
\end{array}\right.\]
Note that $\psi_\delta\in\dom c_\varepsilon$.  
Some calculations show that
\begin{align*}
c_\varepsilon(\psi_\delta) - \lambda_{\varepsilon,1}(0) \int_\Lambda |\psi_\delta|^2 \ds \dt  =
&
\int_\Lambda \frac{|\partial_s \psi_\delta|^2}{f_\varepsilon^2}\ds \dt
+\frac{1}{4}\int_\Lambda \frac{\left(\partial_s f_\varepsilon\right)^2}{f_\varepsilon^4}|\psi_\delta|^2\ds \dt 
- {\cal R} \int_{\Lambda} \frac{\partial_s f_\varepsilon}{f_\varepsilon^3} \overline{\psi_\delta}\,\partial_s \psi_\delta \ds \dt \\
& + 
\frac{1}{\varepsilon^2} \int_\Lambda|\partial_t \psi_\delta|^2\ds\dt - \lambda_{\varepsilon, 1}(0) \int_\Lambda |\psi_\delta|^2 \ds \dt 
+ \int_\Lambda V_\varepsilon |\psi_\delta|^2 \ds \dt \\
= &
\int_\Lambda \frac{|\partial_s \psi_\delta|^2}{f_\varepsilon^2}\ds \dt 
+\frac{1}{4}\int_\Lambda \frac{\left(\partial_s f_\varepsilon\right)^2}{f_\varepsilon^4}|\psi_\delta|^2\ds \dt
- {\cal R} \int_{\Lambda} \frac{\partial_s f_\varepsilon}{f_\varepsilon^3} \overline{\psi_\delta}\,\partial_s \psi_\delta \ds \dt \\
& +
\int_\Lambda (V_\varepsilon -  Y_\varepsilon^0) |\psi_\delta|^2 \ds \dt.
\end{align*}

Now, note that $f_\varepsilon \to 1$, $\partial_s f_\varepsilon \to 0$,
$(V_\varepsilon - Y_\varepsilon^0) \to (|\Theta'|^2 - \gamma^2)/2$, uniformly, as $\varepsilon \to 0$.
Then,
\begin{equation*}
c_\varepsilon(\psi_\delta) -\lambda_{\varepsilon,1}(0)\|\psi_\delta\|^2
		 \to \delta 
		+ \frac{1}{2}\int_{-s_0}^{s_0}  \left(|\Theta'(s)|^2 - \gamma^2\right) \ds,
	\end{equation*}
as $\varepsilon \to 0$. Since 
$\|\psi_\delta\|^2 = 2s_0+\delta^{-1}$,
one has
\begin{equation}\label{indeltadisc}
\frac{c_\varepsilon(\psi_\delta)-\lambda_{\varepsilon,1}(0)\|\psi_\delta\|^2}{\|\psi_\delta\|^2} 
\to 
O(\delta^2) +
\frac{\delta}{2} \int_{-s_0}^{s_0} \left(|\Theta'(s)|^2 - \gamma^2\right) \ds,
\end{equation}
as $\varepsilon \to 0$.

Since $\int_{-s_0}^{s_0} (|\Theta'(s)|^2 - \gamma^2)\ds < 0$,  we can choose
$\delta$ small enough so that the limit in (\ref{indeltadisc}) is negative. Consequently,
there exists $\varepsilon_1>0$ so that 
\[\frac{c_\varepsilon(\psi_\delta)-\lambda_{\varepsilon,1}(0)\|\psi_\delta\|^2}{\|\psi_\delta\|^2} < 0,\]
for all $\varepsilon \in (0, \varepsilon_1)$.
\end{proof}

\begin{proof}[\bf Proof of Theorem \ref{casetheodisspec}]
Given $\delta>0$ and $\eta>0$, define 
$\psi_{\delta,\eta}(s,t) := \phi_\eta(s) u_{\varepsilon,1}^0(t),$
where	
\[\phi_\eta(s) :=\left\{\begin{array}{ll}
e^{\delta (s+s_0)},  & s\leq -s_0,\\
1 +\eta\, (\gamma - |\Theta'(s)|), & -s_0\leq s \leq s_0,\\
e^{-\delta (s-s_0)},  & s\geq s_0.
\end{array}\right.\]
Note that $\psi_{\delta,\eta}\in\dom c_\varepsilon.$ Similarly as in the proof of Theorem \ref{intmaintheodisspec},  we can show that
\begin{equation*}
c_\varepsilon(\psi_{\delta,\eta}) -\lambda_{\varepsilon,1}(0)\|\psi_{\delta,\eta}\|^2	
\to
\delta 
+ O(\eta^2)
-\eta \int_{-s_0}^{s_0} \left(|\Theta'(s)| - \gamma\right)^2 \left(|\Theta'(s)| + \gamma\right){\rm d}s,
\end{equation*}
as $\varepsilon\to0.$ 
Since $\|\psi_{\delta,\eta}\|^2 = 2s_0+\delta^{-1} +  O(\eta^2),$
one has
\begin{equation} \label{indeltadisc2}
\frac{c_\varepsilon(\psi_{\delta,\eta})-\lambda_{\varepsilon,1}(0)\|\psi_{\delta,\eta}\|^2}{\|\psi_{\delta,\eta}\|^2} 
\to
O(\delta^2) + \delta  O(\eta^2) - \delta  \eta \int_{-s_0}^{s_0} \left(|\Theta'(s)| - \gamma\right)^2 \left(|\Theta'(s)| + \gamma\right)\ds,
\end{equation}
as $\varepsilon\to0.$
Taking $\eta = \sqrt{\delta},$ again we can choose $\delta$ small enough so that the limit in (\ref{indeltadisc2}) is negative.  
Then, there exists $\varepsilon_2>0$ so that 
\begin{equation*}
\frac{c_\varepsilon(\psi_{\delta,\eta})-\lambda_{\varepsilon,1}(0)\|\psi_{\delta,\eta}\|^2}{\|\psi_{\delta,\eta}\|^2} 
<0,
\end{equation*}
for all $\varepsilon\in(0,\varepsilon_2)$.
\end{proof}

\section{Thin strips}\label{sectionthinstrips}

In this section we present the proof of Theorem \ref{theothinstripdiv} stated in the Introduction.
The strategy will be to establish upper and lower bounds for the eigenvalues
$\lambda_j(-\tilde{\Delta}_\varepsilon)$. 
Recall $-\tilde{\Delta}_\varepsilon$ is the self-adjoint operator
associated with the quadratic form $\tilde{a}_\varepsilon(\varphi)$; see (\ref{atildintsecd}).
Then, the analysis will be based on estimates for $\tilde{a}_\varepsilon(\varphi)$.

Define
\[\tilde{f}_\varepsilon(s,t):= \sqrt{1+ |\Theta_\varepsilon'(s)|^2 \varepsilon^2 t^2},\]
and consider the Hilbert space 
${\cal H}_\varepsilon := L^2(\Lambda, \tilde{f}_\varepsilon\ds \dt)$; the norm in this space is denoted by
$\| \cdot \|_{\cal H_\varepsilon}$.
Performing a change of coordinates similar to that in Section \ref{secaochange}, 
$\tilde{a}_\varepsilon(\varphi)$ becomes
\[\tilde{b}_\varepsilon(\psi) :=  
\int_\Lambda \frac{|\partial_s \psi|^2}{\tilde{f}_\varepsilon} \ds \dt
+ \frac{1}{\varepsilon^2} \int_\Lambda  |\partial_t \psi|^2 \tilde{f}_\varepsilon \ds \dt,\]
$\dom \tilde{b}_\varepsilon = H_0^1(\Lambda) \subset {\cal H}_\varepsilon$.

\vspace{0.3cm}
\noindent
{\bf Upper bound.}
Denote by $\chi_1(t) := \cos(\pi t/2)$ the first eigenfunction of the Dirichlet Laplacian $-\Delta_{(-1,1)}^D$ in $L^2(-1,1)$;
$(\pi/2)^2$ is the eigenvalue associated with $\chi_1$.
Consider the subspace 
\[{\cal A}_\varepsilon:=\{\varphi_w:=w(s)\chi_1(t) (\tilde{f}_\varepsilon(s,t))^{-1/2}: w \in H^1(\mathbb R)\}\]
of the Hilbert space
${\cal H}_\varepsilon$. 
The identification $w \mapsto \varphi_w$, $w \in H^1(\mathbb R)$, motivates the definition of the one-dimensional
quadratic form
\[m_\varepsilon(w) := \tilde{b}_\varepsilon(\varphi_w) - (\pi/2\varepsilon)^2 \| \varphi_w\|_{{\cal H}_\varepsilon}^2,\]
$\dom m_\varepsilon := H^1(\mathbb R)$. Denote by $M_\varepsilon$ the self-adjoint operator associated with
$m_\varepsilon (w)$.
In particular, for each $j \in \mathbb N$,
\begin{equation}\label{ineigfir}
\lambda_j(-\tilde{\Delta}_\varepsilon) - \left( \frac{\pi}{2 \varepsilon}  \right)^2 \leq \lambda_j(M_\varepsilon). 
\end{equation}
We are going to get upper bounds for the values $\lambda_j(M_\varepsilon)$.

Define the function
\[ W_\varepsilon(s,t) := 
- \frac{7 |(\Theta_\varepsilon' \cdot \Theta_\varepsilon'')(s)|^2 \varepsilon^4 t^4}{4\tilde{f}_\varepsilon^6 (s,t)}
+ \frac{(2|\Theta_\varepsilon''(s)|^2 + 2(\Theta_\varepsilon' \cdot \Theta_\varepsilon''')(s) - 3 |\Theta_\varepsilon'(s)|^4 )\varepsilon^2 t^2}{4 \tilde{f}_\varepsilon^4 (s,t)}
%
%-\frac{3 |\Theta_\varepsilon'(s)|^4 \varepsilon^2 t^2}
%{4\tilde{f}_\varepsilon^4 (s,t)}  
+\frac{|\Theta_\varepsilon'(s)|^2}{2\tilde{f}_\varepsilon^2 (s,t)}.
\]
Recall that we have the condition (V) in the Introduction, we get the estimates
\[\|(1/\tilde{f}_\varepsilon^2) - 1 \|_{{L^\infty(\Lambda)}} \leq  \varepsilon^{2-2a}, \quad 
%\|\partial_s \tilde{f}_\varepsilon/ \tilde{f}_\varepsilon^2\|_\infty \leq \varepsilon^{2-(a+b)},
%\quad
%\|\partial_s(\partial_s \tilde{f}_\varepsilon/ \tilde{f}_\varepsilon^2)\|_{L^\infty} 
%\leq K (\varepsilon^{2-2b}+ \varepsilon^{2-(a+c)} + \varepsilon^{4-2(a+b)}),\]
%
\| W_\varepsilon 
- |\Theta'_\varepsilon|^2/2\|_{L^\infty(\Lambda)}
\leq 
K (\varepsilon^{4-2(a+b) } + \varepsilon^{2-2b}+ \varepsilon^{2-(a+c)} + \varepsilon^{2-4a}),\]
for some $K> 0$, for all $\varepsilon > 0$ small enough.
Finally, some calculations show that
\begin{align*}
m_\varepsilon(w)  
& =
\int_\Lambda \frac{|w'\chi_1|^2}{\tilde{f}_\varepsilon^2} \ds\dt
%+\frac{1}{4}\int_\Lambda \frac{(\partial_s \tilde{f}_\varepsilon)^2}{\tilde{f}_\varepsilon^4} |w\chi_1|^2 \ds \dt
%- {\cal R}\int_{\Lambda} \frac{\partial_s \tilde{f}_\varepsilon}{\tilde{f}_\varepsilon^2}
%\overline{w} w' |\chi_1|^2 \ds \dt 
+\int_\Lambda W_\varepsilon |w \chi_1|^2 \ds \dt \\
%& =
%\int_\Lambda \frac{|w'\chi_1|^2}{\tilde{f}_\varepsilon^2} \ds\dt
%+\frac{1}{4}\int_\Lambda \frac{(\partial_s \tilde{f}_\varepsilon)^2}{\tilde{f}_\varepsilon^4} |w\chi_1|^2 \ds \dt
%+ \frac{1}{2}\int_{\Lambda} \partial_s \left( \frac{\partial_s \tilde{f}_\varepsilon}{\tilde{f}_\varepsilon^2} \right)
%|w \chi_1|^2 \ds \dt 
%+\int_\Lambda \widetilde{V}_\varepsilon |w \chi_1|^2 \ds \dt \\
%& \leq 
%\int_\Lambda |w'\chi_1|^2\ds\dt
%+ O(\varepsilon^{4-2(a+b)}) \int_\mathbb R |w \chi_1|^2 \ds \dt
%+ O(\varepsilon^{d}) \int_{\Lambda} |w \chi_1|^2 \ds \dt  %\\
%&
%+ O(\varepsilon^{2-4a})\int_\Lambda |w \chi_1|^2 \ds \dt
%+\int_\Lambda \frac{|\Theta_\varepsilon'(s)|^2}{2}|w \chi_1|^2 \ds \dt \\
& \leq 
\int_\mathbb R \left(|w'|^2 + 
\frac{|\Theta_\varepsilon'(s)|^2}{2}|w|^2 \right) \ds 
+ O(\varepsilon^{d}) \int_\mathbb R |w|^2 \ds,
\end{align*}
for all $w \in H^1(\mathbb R)$, for all $\varepsilon > 0$ small enough, where
$d = \min\{ 4-2(a+b), 2-2b, 2-(a+c), 2 - 4a\}$.  As a consequence, for each $j \in \mathbb N$, 
\begin{equation}\label{v7mvarineq}
\lambda_j(M_\varepsilon) \leq 
\lambda_j \left(-\Delta_\mathbb R  + \frac{|\Theta'_\varepsilon(s)|^2}{2} \Id \right) + O(\varepsilon^{d}).
\end{equation}
By (\ref{ineigfir}) and (\ref{v7mvarineq}), for each $j \in \mathbb N$,
\begin{equation}\label{upperinvsev}
\lambda_j(-\tilde{\Delta}_\varepsilon) - \left( \frac{\pi}{2 \varepsilon}  \right)^2  \leq 
\lambda_j \left(-\Delta_\mathbb R  + \frac{|\Theta'_\varepsilon(s)|^2}{2} \Id \right) + O(\varepsilon^{d}).
\end{equation}

\vspace{0.3cm}
\noindent
{\bf Lower bound.} 
For each $\varepsilon \geq 0$ small enough, consider  
the one-dimensional  self-adjoint operator
\[(S_\varepsilon v)(t) : = -v''(t) - \frac{\varepsilon t}{1+\varepsilon t^2} v'(t), \quad \dom S_\varepsilon = H_0^1(-1,1),\]
acting in the Hilbert space $L^2((-1,1), \sqrt{1 + \varepsilon t^2} \dt)$.
The particular case  $\varepsilon = 0$ corresponds to the Dirichlet Laplacian operator $-\Delta_{(-1,1)}^D$ 
in $L^2(-1,1)$.

Denote by $\Sigma(\varepsilon)$ the first eigenvalue of $S_\varepsilon$.
By the analytic perturbation theory, we can write 
\[\Sigma(\varepsilon) = \left(\frac{\pi}{2}\right)^2 + \delta(\varepsilon) \varepsilon + O(\varepsilon^2),\]
where
\[\delta(\varepsilon):= - \int_{-1}^1 \frac{t}{\sqrt{1+\varepsilon t^2}} \chi_1'(t) \chi_1(t) \dt;\]
see \cite{katotosio} for more details.
Consequently, for each $\psi \in H_0^1(\Lambda)$, we have the estimate
\[\tilde{b}_\varepsilon(\psi) \geq 
\int_\Lambda \left( \frac{|\partial_s \psi|^2}{\tilde{f}_\varepsilon} + 
\frac{\Sigma(\varepsilon^2 |\Theta_\varepsilon'(s)|^2)}{\varepsilon^2}  \tilde{f}_\varepsilon 
|\psi|^2 \right) \ds \dt.\]
More exactly,
\begin{equation}\label{ineqpertanaly}
\tilde{b}_\varepsilon(\psi) - \left( \frac{\pi}{2 \varepsilon} \right)^2 \| \psi  \|_{{\cal H}_\varepsilon}^2 
\geq  
%\int_\Lambda \left(\frac{|\partial_s \psi|^2}{\tilde{f}_\varepsilon} + 
%|\Theta'_\varepsilon(s)|^2 \delta(\varepsilon^2 |\Theta_\varepsilon'(s)|^2) \tilde{f}_\varepsilon |\psi|^2 
%+ O(\varepsilon^2 |\Theta'_\varepsilon(s)|^4) \tilde{f}_\varepsilon |\psi|^2  \right) \ds \dt \nonumber\\
%\geq & 
\int_\Lambda \left(\frac{|\partial_s \psi|^2}{\tilde{f}_\varepsilon} + 
|\Theta'_\varepsilon(s)|^2 \delta(\varepsilon^2 |\Theta_\varepsilon'(s)|^2) \tilde{f}_\varepsilon |\psi|^2   \right) \ds \dt + O(\varepsilon^{2-4a})\|\psi\|_{\cal H_\varepsilon}^2.
\end{equation}

Now, define the quadratic form
\[n_\varepsilon(\psi) := 
\int_\Lambda \left(\frac{|\partial_s \psi|^2}{\tilde{f}_\varepsilon} + 
|\Theta_\varepsilon'(s)|^2 \delta(\varepsilon^2 |\Theta_\varepsilon'(s)|^2) \tilde{f}_\varepsilon |\psi|^2   \right)\ds \dt,\]
$\dom n_\varepsilon = H_0^1(\Lambda)$. Denote by $N_\varepsilon$ the self-adjoint operator
associated with $n_\varepsilon(\psi)$.
For each $j \in \mathbb N$, inequality (\ref{ineqpertanaly}) implies
\begin{equation}\label{mvarleqadd}
\lambda_j(N_\varepsilon) + O(\varepsilon^{2-4a}) \leq 
\lambda_j(-\tilde{\Delta}_\varepsilon) - \left( \frac{\pi}{2 \varepsilon} \right)^2.
\end{equation}
The next step is to find lower bounds for the values $\lambda_j(N_\varepsilon)$.

\begin{Lemma}\label{lemKanaly}
There exists a number $K > 0$ so that
\[ \left\|  |\Theta'_\varepsilon(s)|^2 \delta(\varepsilon^2 |\Theta_\varepsilon'(s)|^2) \tilde{f}_\varepsilon   -
\frac{|\Theta'_\varepsilon(s)|^2}{2}  
\right\|_{L^\infty(\Lambda)} \leq K \varepsilon^{1-3a},\]
for all $\varepsilon > 0$ small enough.
\end{Lemma}
\begin{proof}
At first, note that
\begin{align*}
\left|  |\Theta'_\varepsilon(s)|^2 \delta(\varepsilon^2 |\Theta_\varepsilon'(s)|^2) \tilde{f}_\varepsilon   -
\frac{|\Theta'_\varepsilon(s)|^2}{2}  
\right| 
& \leq
|\Theta'_\varepsilon(s)|^2 \left(\left|  \delta(\varepsilon^2 |\Theta_\varepsilon'(s)|^2)   -
\frac{1}{2}  
\right| |\tilde{f}_\varepsilon| 
+ \frac{1}{2} |\tilde{f}_\varepsilon - 1| \right),
\end{align*}
for all $(s,t)\in \Lambda$.
Some calculations show that
\begin{equation*}
\delta(\varepsilon^2 |\Theta_\varepsilon'(s)|^2) - \frac{1}{2} 
=  \frac{\pi}{2} \int_{-1}^{1}
\left(\frac{1}{\tilde{f}_\varepsilon} - 1\right) 
t \sin \left(\frac{\pi t}{2}\right) 
\cos\left(\frac{\pi t}{2} \right) \dt.
\end{equation*}
Since $\|(1/\tilde{f}_\varepsilon) -1\|_{L^\infty(\Lambda)} \leq \varepsilon^{1-a},$ 
we have the estimate
$\|\delta(\varepsilon^2 |\Theta_\varepsilon'(s)|^2) - 1/2 \|_{L^\infty(\Lambda)}
\leq \varepsilon^{1-a}/2$. 
Thus, along with the condition $|\Theta_\varepsilon'(s)|\leq K/\varepsilon^a$ and the estimate $\|\tilde{f}_\varepsilon-1\|_{L^\infty(\Lambda)} \leq \varepsilon^{1-a}$, we get the result.
\end{proof}

Using Lemma \ref{lemKanaly} and the estimate $\|1/\tilde{f}_\varepsilon - 1\|_{L^\infty(\Lambda)} \leq \varepsilon^{1-a}$,  one has 
\begin{equation*}
n_\varepsilon(\psi) 
 \geq 
(1+O(\varepsilon^{1-a}))
\int_\Lambda |\partial_s \psi|^2 \ds \dt +
\int_\Lambda  
\frac{|\Theta_\varepsilon'(s)|^2}{2} |\psi|^2 \ds \dt + O(\varepsilon^{1-3a}) \int_\Lambda |\psi|^2 \ds \dt. %\\
%& =
%(1+O(\varepsilon^{2-2a}))
%\int_\Lambda \left( |\partial_s \psi|^2 +
%\frac{|\Theta_\varepsilon'(s)|^2}{2} |\psi|^2 \right) \ds \dt
%+ O(\varepsilon^{2-4a}) \int_\Lambda |\psi|^2 \ds \dt,
\end{equation*}
Since $|\Theta_\varepsilon'(s)| \leq K/\varepsilon^a,$ for all $s\in\mathbb R,$ it follows that
\begin{equation*}
n_\varepsilon(\psi) 
\geq 
(1+O(\varepsilon^{1-a}))
\int_\Lambda \left( |\partial_s \psi|^2 +
\frac{|\Theta_\varepsilon'(s)|^2}{2} |\psi|^2 \right) \ds \dt,
\end{equation*}
for all $\psi \in H_0^1(\Lambda)$, for all $\varepsilon > 0$ small enough.

As a consequence, for each $j \in \mathbb N$,
\begin{equation}\label{inv7lowerb}
(1+O(\varepsilon^{1-a})) \lambda_j\left(-\Delta_\mathbb R + \frac{|\Theta_\varepsilon'(s)|^2}{2} \Id\right) 
\leq
\lambda_j(N_\varepsilon).
\end{equation}

Inequalities (\ref{mvarleqadd}) and  (\ref{inv7lowerb}) ensure that, for each $j \in \mathbb N$, 
\begin{equation}\label{ineqlower}
(1+O(\varepsilon^{1-a})) \lambda_j\left(-\Delta_\mathbb R  + \frac{|\Theta'_\varepsilon(s)|^2}{2} \Id \right) + O(\varepsilon^{2-4a})
\leq
\lambda_j(-\tilde{\Delta}_\varepsilon) - \left(\frac{\pi}{2 \varepsilon}\right)^2.
\end{equation}

\vspace{0.3cm}

Due to inequalities (\ref{upperinvsev}) and (\ref{ineqlower}), we will study the spectral problem
of the operator $-\Delta_\mathbb R + (|\Theta_\varepsilon'(s)|^2/2) \Id$.

\begin{Proposition}\label{proprooftheoint02}
For each $j \in \mathbb N$,
\begin{equation}\label{thetainfsupps}
\lim_{\varepsilon \to 0} \lambda_j\left(-\Delta_\mathbb R  + \frac{|\Theta'_\varepsilon(s)|^2}{2} \Id \right) =
\lambda_j\left(-\Delta_\mathbb R  + \frac{|\Theta'(s)|^2}{2} \Id \right).
\end{equation}
\end{Proposition}
\begin{proof}
Recall the conditions (I)-(VII) in the Introduction.
Consider the quadratic forms
\[y (w) := 
\int_\mathbb R \left( |w'|^2 + \frac{|\Theta'(s)|^2}{2} |w|^2 \right)\ds,
\quad \dom y = \{w \in H^1(\mathbb R): y(w) < \infty\},\]
and, for each $\varepsilon > 0$ small enough, 
\[y_\varepsilon(w) := 
\int_\mathbb R \left( |w'|^2 + \frac{|\Theta_\varepsilon'(s)|^2}{2} |w|^2 \right)\ds,
\quad \dom y_\varepsilon = H^1(\mathbb R);\]
denote by $Y$ and $Y_\varepsilon$ the self-adjoint operators associated with $y(w)$ and $y_\varepsilon(w)$, respectively. In particular,
the condition (III) implies that $|\Theta_\varepsilon'(s)| \to |\Theta'(s)|$ pointwise,
as $\varepsilon \to 0$. Thus, for each  $w \in C_0^\infty(\mathbb R)$,
$Y_\varepsilon w \to  Y w$, as $\varepsilon \to 0$.
Since $C_0^\infty(\mathbb R)$ is a core of $Y$,
given a constant $c> 0$,
\begin{equation}\label{propsolomy001}
(Y_\varepsilon + c \Id)^{-1} u \to (Y+ c \Id)^{-1} u, \quad \forall u \in L^2(\mathbb R).
\end{equation}

Now, denote by ${\cal C}$ the ideal of all compact operators
in the algebra of all bounded operators in $L^2(\mathbb R)$. Fix $\varepsilon^* > 0$ small enough. By (\ref{conthetdernu12}) and  (V), one has

\begin{equation}\label{propsolomy002}
{\rm dist} ((Y_{\varepsilon^*} + c \Id)^{-1}, {\cal C}) \leq \left(\liminf_{|s| \to \infty} \frac{|\Theta'_{\varepsilon^*}(s)|^2}{2} \right)^{-1} 
\leq 2 (\varepsilon^*)^{2a}.
\end{equation}
The condition (VII) implies that
\begin{equation}\label{propsolomy003}
(Y_\varepsilon + c \Id)^{-1} \leq (Y_{\varepsilon^{*}}+c\Id)^{-1},
\end{equation}
for all $\varepsilon < \varepsilon^*$.

As a consequence of (\ref{propsolomy001}), (\ref{propsolomy002}) and (\ref{propsolomy003}),
Proposition 5.3 of \cite{solomyak}  ensures that 
\[\limsup_{\varepsilon \to 0} \|(Y_\varepsilon + c \Id)^{-1} - (Y + c \Id)^{-1}\| \leq 2 (\varepsilon^*)^{2a}.\]
Taking $\varepsilon^* \to 0$, 
\[\|(Y_\varepsilon + c \Id)^{-1} - (Y + c \Id)^{-1}\| \to 0, \quad \hbox{as} \quad \varepsilon \to 0.\] 
Then, we obtain (\ref{thetainfsupps}). 
\end{proof}

\smallskip

\begin{proof}[\bf Proof of Theorem \ref{theothinstripdiv}]
It just to apply inequalities (\ref{upperinvsev}) and (\ref{ineqlower}), and Proposition \ref{proprooftheoint02}.
\end{proof}

\section{Appendix}\label{appendix001}

\subsection*{Stability of the essential spectrum}

The results of this appendix are simple adaptations of Lemma 4.1 and Proposition 4.2 of \cite{davidbriet},
and Lemma 4.2 of \cite{durand}.

\begin{Lemma}\label{caracterizacaoespectro}
A real number $\lambda$ belongs to the essential spectrum of 
$C_\varepsilon$ if and only if there exists a sequence 
$\{\psi_n\}_{n \in \mathbb N} \subset \dom c_\varepsilon$ satisfying the following conditions:
\begin{enumerate}
		\item[$(i)$] $\|\psi_n\| = 1,$ for all $n \in \mathbb N;$
		\item[$(ii)$] $(C_\varepsilon - \lambda\Id) \psi_n \to 0,$ as  $n\to\infty$, in the norm of the dual space 
		$\left(\dom c_\varepsilon\right)^*;$
		\item[$(iii)$] $\operatorname{supp} \psi_n \subset \Lambda\backslash (-n,n)\times(-1,1),$ for all $n \in \mathbb N.$
\end{enumerate} 
\end{Lemma}
\begin{proof}
It is known that $\lambda \in \sigma_{ess}(C_\varepsilon)$ if and only if there exists a sequence
$\{\xi_n\}_{n \in \mathbb N} \subset \dom c_\varepsilon$ satisfying $(i),$ $(ii)$, and
$(iii')$\, $\xi_n \rightharpoonup 0$  in  $L^2(\Lambda)$, as  $n \to 0$;
see, e.g., Theorem 5 in \cite{davidlu}.
Let $\{\psi_n\}_{n \in \mathbb N}$ be a sequence satisfying the conditions $(i),$ $(ii)$ and $(iii)$.
Consequently, it satisfies $(i),$ $(ii)$ and $(iii')$.

Now, let $\{\xi_n\}_{n \in \mathbb N} \subset \dom c_\varepsilon$ be a sequence satisfying
$(i)$, $(ii)$, and $(iii')$.
Take  $\eta\in C^\infty(\mathbb R; \mathbb R)$, $0 \leq \eta \leq 1$, $\eta = 0$ in $[-1,1]$, and $\eta = 1$ in $\mathbb R\backslash(-2,2)$. 
Define the sequence
$\{\eta_k\}_{k \in \mathbb N} \subset C^\infty(\Lambda)$,  where $\eta_k(s,t) := \eta(s/k)$.
Since $(1-\eta_k)(C_\varepsilon+\Id)^{-1}$ is compact in $L^2(\Lambda)$, by $(iii')$, we have
$(1-\eta_k)(C_\varepsilon+\Id)^{-1} \xi_n \to 0$ in
$L^2(\Lambda)$, as $n\to \infty$, for all $k\in \mathbb N$. 
Then there exists a subsequence $\{\xi_{n_k}\}_{k \in \mathbb N}$ of $\{\xi_n\}_{n \in \mathbb N}$ so that
$(1-\eta_k)(C_\varepsilon+\Id)^{-1} \xi_{n_k} \to 0$ in $L^2(\Lambda)$, as $k\to\infty$.
By writing 
$$\xi_{n_k} = (C_\varepsilon + \Id)^{-1} (C_\varepsilon - \lambda \Id)
\xi_{n_k} + (\lambda + 1)(C_\varepsilon + \Id)^{-1} \xi_{n_k},$$
and using $(ii)$, it follows that $(1-\eta_k) \xi_{n_k} \to 0$ in $L^2(\Lambda)$, as $k\to\infty$. 
Thus, we can assume that $\|\eta_k \xi_{n_k}\| \geq 1/2,$ for all $k \in \mathbb N.$ 
Finally, define
\begin{equation*}
\psi_k := \frac{\eta_k \xi_{n_k}}{\|\eta_k \xi_{n_k}\|}, \quad k \in \mathbb N.
\end{equation*}
The sequence $\{\psi_k\}_{k \in \mathbb N} \subset \dom c_\varepsilon$ satisfies the conditions $(i)$ and $(iii)$. 
It remains to verify  $(ii)$, i.e., 
\begin{equation}\label{condicao2}
\sup_{\substack{\phi \in {H}^{1}_{0}(\Lambda)\\ \phi\neq 0}}
\frac{|c_\varepsilon (\phi,\psi_k) -\lambda \langle\phi,\psi_k\rangle|}{\|\phi\|_{+}} \longrightarrow 0,
\end{equation}
as $k\to \infty$,
where $\|\phi\|_+^2 := c_\varepsilon(\phi)+\|\phi\|^2$.

Some calculations show that 
\begin{align*}
c_\varepsilon (\phi,\eta_k \xi_{n_k}) -\lambda \langle\phi,\eta_k\xi_{n_k}\rangle 
= & 
\, c_\varepsilon (\eta_k \phi, \xi_{n_k}) -\lambda \langle\eta_k \phi,\xi_{n_k}\rangle 
+ 
\int_\Lambda \frac{1}{f_\varepsilon^2}\phi\,\partial_s^2\eta_k\, \overline{\xi_{n_k}} \ds \dt \\
& +
2\int_\Lambda \frac{1}{f_\varepsilon^2} 
\left( \partial_s \phi - \frac{\partial_s f_\varepsilon }{f_\varepsilon}\phi\right) \partial_s\eta_k\, \overline{\xi_{n_k}} \ds \dt.
\end{align*}
Since $\{\xi_{n_k}\}_{k \in \mathbb N}$ satisfies $(ii)$, we have
\begin{equation*}
\sup_{\substack{\phi \in {H}^{1}_{0}(\Lambda)\\ \phi\neq 0}}
\frac{| c_\varepsilon (\eta_k \phi, \xi_{n_k}) -\lambda \langle\eta_k \phi,\xi_{n_k}\rangle|}{\|\phi\|_{+}} \,
\leq
\sup_{\substack{\phi \in {H}^{1}_{0}(\Lambda)\\ \eta_k \phi\neq 0}}
\frac{| c_\varepsilon (\eta_k \phi, \xi_{n_k}) -\lambda \langle\eta_k \phi,\xi_{n_k}\rangle |}{\|\eta_k \phi\|_{+}} \longrightarrow 0,
\end{equation*}
as $k \to \infty$.
By H\"older's Inequality and by the estimates $\|\phi\|\leq\|\phi\|_+$ and $c_\varepsilon(\phi)\leq\|\phi\|_+^2$, 
we get
\begin{equation*}
\sup_{\substack{\phi \in {H}^{1}_{0}(\Lambda)\\ \phi\neq 0}}
\left\{\frac{1}{\|\phi\|_{+}} \int_\Lambda \frac{1}{f_\varepsilon^2}|\phi| |\partial_s^2\eta_k| |\xi_{n_k}| \ds \dt \right\} 
\leq \,
\|\partial_s^2 \eta_k\|_{L^\infty(\Lambda)} = k^{-2}\|\eta''\|_{L^\infty(\mathbb R)} \longrightarrow 0, 
\end{equation*}
and 
\begin{equation*}
\sup_{\substack{\phi \in {H}^{1}_{0}(\Lambda)\\ \phi\neq 0}}
\left\{\frac{1}{\|\phi\|_{+}}  
\int_\Lambda \frac{1}{f_\varepsilon^2} 
\left| \partial_s \phi - \frac{\partial_s f_\varepsilon }{2 f_\varepsilon}\phi\right| |\partial_s\eta_k| |\xi_{n_k}|  \ds \dt\right\}\, 
\leq \,
\left\|\partial_s \eta_k\right\|_{L^\infty(\Lambda)} = 	k^{-1}\|\eta'\|_{L^\infty(\mathbb R)} \longrightarrow 0,
\end{equation*}
as $k\to\infty$.
Finally, since $\beta, \beta' \in L^\infty(\mathbb R)$, we have
\begin{equation*}
\sup_{\substack{\phi \in {H}^{1}_{0}(\Lambda)\\ \phi\neq 0}}
\left\{\frac{1}{\|\phi\|_{+}} 
\int_\Lambda
\left|\frac{\partial_s f_\varepsilon }{f_\varepsilon^3}\right|
|\phi| |\partial_s\eta_k| |\xi_{n_k}| \ds \dt\right\}\, 
\leq \,
\left\|\partial_s \eta_k\right\|_{L^\infty(\Lambda)} = 	k^{-1}\|\eta'\|_{L^\infty(\mathbb R)} \longrightarrow 0,
\end{equation*}
as $k\to\infty$.
Then,  (\ref{condicao2}) holds true.
\end{proof}

\begin{Remark}
{\rm The same conclusion of Lemma \ref{caracterizacaoespectro} holds true for the
operator $D_\varepsilon$.}
\end{Remark}

\begin{proof}[\bf Proof of Proposition \ref{proposition1}]
Let $\lambda\in \sigma_{ess}(D_\varepsilon)$, 
then there exists a sequence $\{\psi_n\}_{n \in \mathbb N} \subset \dom d_\varepsilon$ so that 
$\|\psi_n\| = 1,$
$\operatorname{supp} \psi_n \subset \Lambda\backslash (-n,n)\times(-1,1),$
for all $n \in \mathbb N,$
and
$(D_\varepsilon - \lambda\Id) \psi_n \to 0,$ as  $n\to\infty$, in the norm of the dual space 
$\left(\dom d_\varepsilon\right)^*$. 
Recall  $\operatorname{supp} \beta \subset [-s_0, s_0]$, for some $s_0 > 0$. Take $n_0 \in \mathbb N$ so that $n_0 > s_0$.
%Passando $\{\psi_n\}_{n \in \N}$ para a subsequência $\{\psi_n\}_{n=n_0}^\infty$ começando de $n_0$, como  a condição $(iii)$.
For each $\phi \in H^1_0(\Lambda)$, one has
\begin{equation*}
c_\varepsilon (\phi,\psi_n) -\lambda \langle\phi,\psi_n\rangle =
d_\varepsilon (\phi,\psi_n) -\lambda \langle\phi,\psi_n\rangle,
\end{equation*}
for all $n\geq n_0$. Consequently,
\begin{equation*}
\sup_{\substack{\phi \in {H}^{1}_{0}(\Lambda)\\ \phi\neq 0}}
\frac{|c_\varepsilon (\phi,\psi_n) -\lambda \langle\phi,\psi_n\rangle|}{\|\phi\|_+} \longrightarrow 0,
\end{equation*}
as $n\to \infty$.
The Lemma \ref{caracterizacaoespectro}
implies that $\lambda\in \sigma_{ess}(C_\varepsilon)$.
In a similar way, it is possible to show that $\sigma_{ess}(C_\varepsilon) \subset \sigma_{ess}(D_\varepsilon)$.
\end{proof}

\vspace{0.5cm}
\noindent
RAFAEL T. AMORIM, Departamento de Matem\'atica,\\ 
Universidade Federal de S\~ao Carlos - UFSCar, S\~ao Carlos, SP, 13565-905, Brazil\\
e-mail: \url{rafaelamorim@estudante.ufscar.br}

\vspace{0.5cm}
\noindent
ALESSANDRA A. VERRI, Departamento de Matem\'atica, \\
Universidade Federal de S\~ao Carlos - UFSCar, S\~ao Carlos, SP, 13565-905, Brazil\\
e-mail: \url{alessandraverri@ufscar.br}

\end{document}